\documentclass[10pt,reqno]{amsart}

\usepackage{amsfonts}
\usepackage{amsmath}
\usepackage{amssymb}
\usepackage{pstricks,pst-math,pst-xkey}
 \newtheorem{thm}{Theorem}[section]
 \newtheorem{cor}[thm]{Corollary}
 
 \newtheorem{prop}[thm]{Proposition}
 \theoremstyle{definition}
 \newtheorem{defn}[thm]{Definition}
 \theoremstyle{remark}
 \newtheorem{rem}[thm]{Remark}
 \newtheorem{eg}{Example}
 \numberwithin{equation}{section}

\usepackage{graphics}
\usepackage{epsfig}

\begin{document}

\title{Evolutes of plane curves and null curves in Minkowski $3$-space}


\author{Boaventura Nolasco}

\address{Instituto Superior de Ci\^{e}ncias da Educa\c{c}\~{a}o\\
Lubango, Angola.}

\email{beleza2011@live.com.pt}

\author{ Rui Pacheco}
\address{Departamento de Matem\'{a}tica, Universidade da Beira Interior\\
Covilh\~{a}, Portugal.}
\email{rpacheco@ubi.pt}

\begin{abstract}
We use the isotropic projection of Laguerre geometry in order to establish  a correspondence between plane curves and null curves in the Minkowski $3$-space.  We describe the geometry of null curves (Cartan frame, pseudo-arc parameter, pseudo-torsion, pairs of associated curves) in terms of the curvature of the corresponding plane curves. This leads to an alternative description of all plane curves which are Laguerre congruent to a given one.

\end{abstract}

\maketitle

\section{Introduction}
Two dimensional Laguerre geometry is the geometry of oriented contact between circles in the Euclidean plane $\mathbf{E}^2$ (points of $\mathbf{E}^2$ are viewed as circles with radius zero).   An efficient and intuitive model of Laguerre geometry is the so called Minkowski model \cite{cecil,FS}. This model establishes a $1$-$1$ correspondence, called the isotropic projection \cite{cecil,FS},  between points in the Minkowski $3$-space $\mathbf{E}_1^3$ and oriented circles in the Euclidean plane. The group of Laguerre transformations is identified with the subgroup of the group of affine transformations of $\mathbf{E}_1^3$ which is generated by linear (Lorentzian) isometries, translations and dilations of $\mathbf{E}_1^3$.

A differentiable one-parameter family of circles in the Euclidean plane corresponds, via  isotropic projection, to a differentiable curve in the Minkowski $3$-space. In particular, the family of osculating circles to a given plane curve $\gamma$ corresponds to a certain  curve $\varepsilon$ in $\mathbf{E}_1^3$, which we will call the \emph{$L$-evolute} of $\gamma$. The $L$-evolute of a plane curve is a null curve and, conversely, any null curve is the $L$-evolute of some plane curve (see Section \ref{levolutesection}).

The importance of light like (degenerate) submanifolds in relativity theory has been emphasized by many mathematical and physical researchers (see the monographs \cite{Dugal1,Dugal2}). Several authors have also investigated  the particular case of null curves in the Minkowski $3$-space  (see the surveys \cite{ino-lee,rafael}). In the present paper we describe the geometry of null curves (Cartan frame, pseudo-arc parameter, pseudo-torsion) in terms of the curvature of the corresponding plane curves. This leads us to the notion of \emph{potential function} (see Definition \ref{defpotential}). A potential function together with an initial condition completely determines the null curve and the underlying plane curve through the formulae of Theorem \ref{fund}. As a consequence, we obtain an alternative description of all plane curves which are \emph{Laguerre congruent}  to a given curve $\gamma$ as follows (see Remark \ref{lagcongruent} for details): starting with the plane curve $\gamma$, compute the pseudo-torsion $\tau$ of its $L$-evolute; up to scale, the pseudo-torsion is invariant under Laguerre transformations of $\gamma$ and the potential functions associated to null curves with pseudo-torsion $\tau$ are precisely the solutions of a certain second order differential equation; use formulae of Theorem \ref{fund} to recover from these solutions the corresponding plane curves; in this way, we obtain all the plane curves which are Laguerre congruent with $\gamma$.

 In Section \ref{assoc}, we describe, in terms of their potential functions, some classes of associated null curves: Bertrand pairs  \cite{BBI,ino-lee}, null curves with common binormal lines \cite{H-I}, and binormal-directional curves \cite{choikim}. We shall also prove the following two interesting results on null curves:
1) a  null helix parameterized by pseudo-arc admits a  null curve parameterized by pseudo-arc with common binormal lines at corresponding points if, and only if, it has pseudo-torsion $\tau=0$ (see Corollary \ref{corolint}); 2) given a null curve $\varepsilon$  parameterized by pseudo-arc, there exists a null helix $\bar\varepsilon$  parameterized by pseudo-arc with constant pseudo-torsion $\bar\tau=0$  and a $1$-$1$ correspondence between points of the two curves $\varepsilon$ and $\bar\varepsilon$ such that, at corresponding points, the  tangent lines are parallel (see Theorem \ref{finish}).
\vspace{.10in}

\section{Curves in the Euclidean plane}
We start by fixing some notation and by recalling  standard facts concerning curves in the Euclidean plane.

\vspace{.10in}

Let $I$ be an open interval of $\mathbb{R}$ and $\gamma:I\to \mathbf{E}^2$ be a regular plane curve, that is, $\gamma$ is differentiable sufficiently many times and $\dot{\gamma}(t)\neq 0$ for all $t\in I$, where $\mathbf{E}^2=(\mathbb{R}^2,\cdot)$ and $\cdot$ stands for  the standard Euclidean inner product on $\mathbb{R}^2$.
Consider the unit tangent vector $\mathbf{{t}}=\dot{\gamma}/|\dot{\gamma}(t)|$  and the unit normal vector $\mathbf{n}=J(\mathbf{{t}})$ of $\gamma$, where $J$ is the anti-clockwise rotation by $\pi/2$. The pair $(\mathbf{{t}},\mathbf{{n}})$ satisfies the Frenet equations
$\dot{\mathbf{t}}=|\dot{\gamma}|k\mathbf{n}$ and $\dot{\mathbf{n}}=-|\dot{\gamma}|k\mathbf{t}$,
where $k=\det(\dot{\gamma},\ddot{\gamma})/|\dot{\gamma}|^3$ is the curvature of $\gamma$.
We denote by $l_{\gamma,t_0}(t)$ the arc length at $t$ of   $\gamma$ with respect to the starting point $t_0\in I$.

Assume that the curvature $k$ is a nonvanishing function on $I$. Let  $u(t)=1/k(t)$ be the (signed) radius of curvature at $t$. The \emph{evolute} of $\gamma$ is the curve $\gamma_\varepsilon:I\to \mathbb{R}^2$ defined by
$\gamma_\varepsilon(t)=\gamma(t)+ u(t){\mathbf{n}(t)},$  which is a regular curve if, additionally, $\dot{u}$ is a nonvanishing function on $I$. In this case, the curvature
$k_\varepsilon$ of $\gamma_\varepsilon$ is given by
\begin{equation}\label{curvatureevolute} k_\varepsilon=\frac{\det(\dot{\gamma}_\varepsilon,\ddot{\gamma}_\varepsilon)}{|\dot{\gamma}_\varepsilon|^3}=\frac{|\dot{\gamma}|}{u\dot{u}}.
\end{equation}
We also have
\begin{equation*}\label{arclengthevolute}
 l_{\gamma_\varepsilon,t_0}(t)=\int_{t_0}^t|\dot{\gamma_\varepsilon}(s)|ds=\pm\int_{t_0}^t\dot{u}(s)ds=\pm\big(u(t)-u(t_0)\big),
 \end{equation*}
which means that $u$ is an arc length parameter of the evolute.

If $t$ is an arc length parameter of $\gamma$, the curve $\gamma_\iota(t)=\gamma(t)-t\mathbf{t}(t)$ is an \emph{involute} of $\gamma$.  It is well known that the evolute of the involute $\gamma_\iota$ is precisely $\gamma$ and that, for any other choice of arc length parameter, the corresponding involute of $\gamma$ is parallel to $\gamma_\iota$.

Finally, recall that,  given a smooth function $k:I\to \mathbb{R}$, there exists a plane curve $\gamma:I\to \mathbb{R}^2$ parameterized by  arc length whose curvature is $k$. Moreover, $\gamma$ is unique up to rigid motion and is given by
$\gamma(t)=(\int\cos\theta(t)dt,\int\sin\theta(t)dt)$
where the \emph{turning angle} $\theta(t)$ is given by
$\theta(t)=\int k(t)dt.$

\section{Laguerre geometry and the Minkowski $3$-space}

Consider on $\mathbb{R}^3$ the Lorentzian inner product defined by
$\vec{u}\cdot\vec{v}=u_1v_1+u_2v_2-u_3v_3,$
for $\vec{u}=(u_1,u_2,u_3)$ and $\vec{v}=(v_1,v_2,v_3)$.
The Minkowski $3$-space is the metric space $\mathbf{E}_1^3=(\mathbb{R}^3,\cdot)$. If $\vec{u}\in \mathbf{E}_1^3$ is a spacelike vector, which means that  $\vec{u}\cdot\vec{u}>0$, we denote
 $|\vec{u}|=\sqrt{\vec{u}\cdot\vec{u}}$.
 The light cone $C_P$ with vertex at $P\in\mathbf{E}_1^3$ is the quadric
$$C_P=\{X\in \mathbf{E}_1^3|\,\, (X-P)\cdot (X-P)=0\}.$$ We have a $1$-$1$ correspondence, called the \emph{isotropic projection}, between points in  $\mathbf{E}_1^3$ and oriented circles in the Euclidean plane, which is defined as follows. Given $P=(p_1,p_2,p_3)$ in $\mathbf{E}_1^3$, consider the light cone $C_P$. The intersection of $C_p$ with the Euclidean plane $\mathbf{E}^2\cong\{(u_1,u_2,u_3)\in \mathbf{E}_1^3|\,u_3=0\}$ is a circle centered at $(p_1,p_2)$ with radius $|p_3|$. The orientation of this circle is anti-clockwise if $p_3>0$ and clockwise if $p_3<0$. Points in $\mathbf{E}^2$ are regarded as circles of zero radius and correspond to points $P=(p_1,p_2,p_3)$ in $\mathbf{E}_1^3$ with $p_3=0$. Making use of this correspondence, \emph{Laguerre transformations} are precisely those affine transformations $L:\mathbf{E}_1^3\to\mathbf{E}_1^3$ of the form $L(\vec{u})=\lambda A\vec{u}+\vec{t}$, where $\lambda\in \mathbb{R}\setminus\{0\}$, $\vec t\in\mathbf{E}_1^3$ and $A\in O_1(3)$ is an orthogonal transformation of $\mathbf{E}_1^3$. For details see \cite{cecil}.

\section{Null curves in Minkowski $3$-space}\label{secnul}
 Next we recall some standard facts concerning null curves in the Minkowski $3$-space. For details, we refer the reader to \cite{ino-lee,rafael}.

\vspace{.10in}

A regular curve $\varepsilon:I\to \mathbf{E}_1^3$ with parameter $t$ is called a \emph{null curve} if $\dot{\varepsilon}$ is lightlike, that is $\dot{\varepsilon} \cdot \dot{\varepsilon}= 0$.
Differentiating this, we obtain $\dot\varepsilon\cdot \ddot \varepsilon= 0$, which means that $\ddot \varepsilon$ lies in  $\mathrm{Span}\{\dot{\varepsilon}\}^\perp$.
We will assume throughout this paper that
$\dot\varepsilon(t)$ and $\ddot{\varepsilon}(t)$ are linearly independent for all $t\in I$ (in particular, $\varepsilon$ can not be a straight line). Then we have $\mathrm{Span}\{\dot{\varepsilon}\}^\perp=\mathrm{Span}\{\dot{\varepsilon}, \ddot{\varepsilon}\}$ and $\ddot{\varepsilon}$ is spacelike.
If  $t=\phi(s)$ is a solution of the differential equation
\begin{equation}\label{genpseudoarc}
\frac{d\phi}{ds}=\pm\frac{1}{\sqrt {|\ddot{\varepsilon}\circ \phi|}},
\end{equation}
then $\upsilon=\varepsilon\circ \phi$  satisfies
$\big|\ddot{\upsilon}\big|= 1.$
In this case, we say that
 $s$ is a \emph{pseudo-arc parameter} \cite{F-G-L,ino-lee,rafael} of $\varepsilon$.

Suppose now that the null curve $\varepsilon:I\to \mathbf{E}_1^3$ is parameterized by pseudo-arc parameter $s$, that is $|\ddot\varepsilon|=1$.  The \emph{tangent vector} is $\mathbf{T}=\dot\varepsilon$ and the (unit)  \emph{normal vector} is
$\mathbf{N}=\ddot\varepsilon$. Define also the \emph{binormal vector} as the unique lightlike vector $\mathbf{B}$ orthogonal to $\mathbf{N}$ satisfying $\mathbf{T}\cdot \mathbf{B}\equiv-1$.
The \emph{Cartan frame} $\{\mathbf{T},\mathbf{B},\mathbf{N}\}$ of $\varepsilon$ satisfies the following Frenet equations:
\begin{equation}\label{freneteqspace}
\left(
    \begin{array}{c}
      \dot{\mathbf{T}} \\
      \dot{\mathbf{B}} \\
      \dot{\mathbf{N}} \\
    \end{array}
  \right)=\underbrace{\left(
            \begin{array}{ccc}
              0 & 0 & 1 \\
              0 & 0 & \tau \\
              \tau & 1 & 0 \\
            \end{array}
          \right)}_{A_\tau}\left(
                   \begin{array}{c}
      {\mathbf{T}} \\
    {\mathbf{B}} \\
      {\mathbf{N}} \\
    \end{array}
  \right),
 \end{equation}
where $\tau:I\to\mathbb{R}$ is called the \emph{pseudo-torsion} of $\varepsilon$. It follows  from \eqref{freneteqspace} that \begin{equation}\label{Tc}\dddot{\mathbf{T}}-2\tau \dot{\mathbf{T}}-\dot{\tau}\mathbf{T}=0.\end{equation} On the other hand, since $\varepsilon$ is a null curve parameterized by pseudo-arc length,  we  have $\mathbf{T}\cdot \mathbf{T}=0 $ and  $\dot{\mathbf{T}}\cdot \dot{\mathbf{T}}=1$. Hence the components $f_i$ of $\mathbf{T}=(f_1,f_2,f_3)$ satisfy
\begin{equation}\label{taudet}
 \dddot{f}_i-2\tau \dot{f}_i-\dot{\tau}f_i=0,\qquad f_1^2+f_2^2-f_3^2=0,\qquad \dot f_1^2+\dot f_2^2-\dot f_3^2=1.
\end{equation}

 Conversely, given a function $\tau:I\to \mathbb{R}$, $s_0\in I$  and a fixed basis $\{\mathbf{T}_0,\mathbf{B}_0,\mathbf{N}_0\}$ of
 $\mathbf{E}_1^3$ satisfying
 $$\mathbf{T}_0\cdot \mathbf{T}_0=\mathbf{B}_0\cdot\mathbf{B}_0=0,\quad \mathbf{N}_0\cdot\mathbf{N}_0=1,\quad \mathbf{T}_0\cdot \mathbf{N}_0=\mathbf{B}_0\cdot \mathbf{N}_0=0,\quad \mathbf{T}_0\cdot \mathbf{B}_0=-1,$$
 consider the linear map which is represented,  with respect to this basis, by the matrix $A_\tau$.
This linear map lies in the Lie algebra $\mathfrak{o}_1(3)$, which means that we can integrate in order to get a map $F:I\to O_1(3),$  with $F(s_0)=Id$, such that
the frame $\{\mathbf{T}, \mathbf{B}, \mathbf{N}\}=\{F\mathbf{T}_0, F\mathbf{B}_0, F\mathbf{N}_0\}$
satisfies \eqref{freneteqspace}. Hence $\varepsilon(s)=\int \mathbf{T}(s)ds$
defines a regular null curve parameterized by pseudo-arc length with pseudo-torsion $\tau$, and $\varepsilon$ is unique up to Lorentz isometry.

The pseudo-arc parameter  and pseudo-torsion are preserved under Lorentz isometries. Regarding dilations, we have the following.
\begin{prop}\label{homoex}
 If $\varepsilon$ is a null curve with pseudo-arc parameter $s$ and pseudo-torsion $\tau$, and $\lambda\neq 0$ is a real number, then  $\bar s=\sqrt{|\lambda|}\, s$  is a pseudo-arc parameter of the null curve $\bar \varepsilon= \lambda\,\varepsilon$, which has pseudo-torsion
  \begin{equation}\label{pseudodilation}
  \bar\tau(\bar s)=\frac{1}{|\lambda|}\tau\big({\bar s}/{\sqrt{|\lambda|}}\big).
  \end{equation}
\end{prop}
\begin{proof}
By applying twice the chain rule, we see that the tangent vector $\bar{\mathbf{T}}=\frac{d\bar \varepsilon}{d\bar s}$ and the normal vector  $\bar{\mathbf{N}}=\frac{d^2\bar \varepsilon}{d\bar s^2}$ of $\bar \varepsilon$ are related with the tangent vector $\mathbf{T}$ and the normal vector  $\mathbf{N}$ of $\varepsilon$ by
$$\bar{\mathbf{T}}(\bar s)=\frac{\lambda}{{\sqrt{|\lambda|}}} \mathbf{T}\big({\bar s}/{\sqrt{|\lambda|}}\big),\qquad {\bar{\mathbf{N}}}(\bar s)=\mathrm{sign}(\lambda)\mathbf{N}({\bar s}/{\sqrt{|\lambda|}}). $$
This shows that $\bar s$ is a pseudo-arc parameter. On the other hand, since the pseudo-torsion $\bar \tau$ is the component of  $\dot{\bar{\mathbf{N}}}$ along $\bar{\mathbf{T}}$, formula \eqref{pseudodilation} can now be easily verified.
\end{proof}
%

\section{The $L$-evolute of a plane curve}\label{levolutesection}
A differentiable one-parameter family of circles in the Euclidean plane corresponds, via  isotropic projection, to a differentiable curve in the Minkowski $3$-space $\mathbf{E}_1^3$. In particular, the family of osculating circles to a given plane curve $\gamma$ corresponds to a certain  curve in $\mathbf{E}_1^3$, which we will call the \emph{$L$-evolute} of $\gamma$. In the present section, we show that,  the $L$-evolute of a plane curve is a null curve and that, conversely, any null curve  is the $L$-evolute of some plane curve if it has nonvanishing third coordinate. We  also describe the geometry of null curves in terms of the curvature of the corresponding plane curves.

\vspace{.10in}

\begin{defn}\label{levol} Consider a regular curve $\gamma:I\to \mathbf{E}^2$,  with curvature $k$ and parameter $t$. Assume that $k$ and its derivative $\dot k$ are nonvanishing functions on $I$. Let $\gamma_\varepsilon$ be its evolute and $u=1/k$ its (signed) radius of curvature. The \emph{$L$-evolute} of $\gamma$ is the curve $\varepsilon$ in the Minkowski $3$-space $\mathbf{E}_1^3$ defined by $\varepsilon=(\gamma_\varepsilon,u)$; that is, for each $t$, $\varepsilon(t)\in\mathbf{E}_1^3$ corresponds to the osculating circle of $\gamma$ at $t$ under the isotropic projection.
 \end{defn}

 Observe that whereas the evolute of a curve is independent of the parameterization, the $L$-evolute depends on the orientation: as a matter of fact, if $\gamma_R$ is a orientation reversing reparameterization of $\gamma$, then the trace of the $L$-evolute of $\gamma_R$ is that of $(\gamma_\varepsilon,-u)$.
\begin{prop}\label{null-evolute}
Given a  regular plane curve  $\gamma:I\to \mathbf{E}^2$ in the conditions of Definition \ref{levol}, its $L$-evolute $\varepsilon:I\to \mathbf{E}_1^3$ is a null curve in  $\mathbf{E}_1^3$. Conversely, any null curve $\varepsilon=(\varepsilon_1,\varepsilon_2,\varepsilon_3):I\to \mathbf{E}_1^3$ is the $L$-evolute of some plane curve $\gamma:I\to \mathbf{E}^2$ if $\varepsilon_3$ is nonvanishing on $I$.
\end{prop}
\begin{proof}
  Taking into account the Frenet equations for the regular plane curve $\gamma$, we have $$\dot \varepsilon=(\dot \gamma_\varepsilon,\dot u)=(\dot{\gamma}+\dot{u}\mathbf{n}+u\dot{\mathbf{n}},\dot u )=\dot{u}(\mathbf{n},1).$$
  Hence $\dot\varepsilon \cdot \dot\varepsilon=0$, since $(\mathbf{n},1)$ is a lightlike vector. Moreover, $\varepsilon$ is regular, since $\dot u=-\dot k/k^2\neq 0$ on $I$. Hence $\varepsilon$ is a null curve.

  Conversely, let $\varepsilon$ be a null curve in $\mathbf{E}_1^3$ with parameter $t$.  We must have $\dot\varepsilon_3(t)\neq 0$ for all $t$: as a matter of fact, if $\dot\varepsilon_3(t_0)=0$ for some $t_0$, then $\dot\varepsilon_1(t_0)^2+\dot\varepsilon_2(t_0)^2=0,$ since $\varepsilon$ is null;
 hence $\dot\varepsilon(t_0)=0$ and $\varepsilon$ is not regular, which is a contradiction.
   This means that we can reparameterize $\varepsilon$ with the parameter $u=\varepsilon_3(t)$ and  assume that  $\varepsilon$ is of the form
  $\varepsilon(u)=(\varepsilon_1(u),\varepsilon_2(u),u)$. By hypothesis, $u\neq 0$. Consider the plane curve $\gamma$  defined by $\gamma(u)=\gamma_\varepsilon(u)-u\dot\gamma_\varepsilon(u),$ where $\gamma_\varepsilon=(\varepsilon_1,\varepsilon_2)$.  Observe that $\gamma_\varepsilon$ is a unit speed curve:  since  $\dot\varepsilon(u)=(\dot\varepsilon_1(u),\dot\varepsilon_2(u),1)$ and  $\dot\varepsilon\cdot \dot\varepsilon=0$, we get $|\dot\gamma_\varepsilon|=1$.
 Hence $\gamma$ is an involute of $\gamma_\varepsilon$; and, consequently, $\gamma_\varepsilon$ is the evolute of $\gamma$. In particular, the curvature $k$ of $\gamma$ satisfies $k=\pm1/u$.
If $k=1/u$, then $\varepsilon$ is the $L$-evolute of $\gamma$. If  $k=-1/u$, then $\varepsilon$ is the $L$-evolute of a orientation reversing reparameterization $\bar{\gamma}$ of $\gamma$.

\end{proof}

To make it clear, throughout the rest of this paper, we will assume that
\begin{enumerate}
  \item[a)] all plane curves are regular, with nonvanishing $k$ and $\dot k$;
  \item[b)] all null curves, and the $L$-evolutes in particular, are such that $\dot\varepsilon$ and $\ddot\varepsilon$ are everywhere linearly independent -- in particular, we are excluding null straight lines in $\mathbf{E}_1^3$ and we can always reparameterize by pseudo-arc.
\end{enumerate}

\begin{prop}
 Let $t$ be an arc length parameter of $\gamma$ and $\varepsilon$ the $L$-evolute of $\gamma$. Let $s$ be a pseudo-arc parameter of $\varepsilon$, where $t=\phi(s)$ and $\phi$ is a solution of \eqref{genpseudoarc}.  Then
\begin{equation}\label{ts}
\dot\phi(s)=\pm\sqrt{\Big|\frac{u(\phi(s))}{\dot u(\phi(s))}\Big|},
\end{equation}
where $u=1/k$ is the radius of curvature of $\gamma$.
\end{prop}
\begin{proof}
By hypothesis, the null curve  $\upsilon=\varepsilon\circ \phi$ satisfies $|\ddot\upsilon|=1$. Taking into account the Frenet equations for plane curves and the chain rule, we have
\begin{equation}\label{n1}
\dot\upsilon(s)=\dot\phi(s)\dot\varepsilon(\phi(s))=\dot\phi(s)\dot u(\phi(s))\big(\mathbf{n}(\phi(s)),1\big)\end{equation}
and
\begin{align}\nonumber\ddot{\upsilon}(s)=\big\{\ddot{\phi}(s)\dot{u}(\phi(s))+\dot\phi(s)^2\ddot{u}&(\phi(s))\big\}\big(\mathbf{n}(\phi(s)),1\big)\\&-
\dot\phi(s)^2k(\phi(s))\dot{u}(\phi(s))\big(\mathbf{t}(\phi(s)),0\big)\label{n2}\end{align}
Since $(\mathbf{n},1)$  is a lightlike vector orthogonal to $(\mathbf{t},0)$  it follows that
\begin{equation}\label{n3}
1=|\ddot{\upsilon}(s)|=\dot\phi(s)^4\big(k(\phi(s))\dot{u}(\phi(s))\big)^2,\end{equation}
from which we deduce \eqref{ts}.
\end{proof}

\begin{defn}
Two plane curves $\gamma$ and $\bar{\gamma}$, with $L$-evolutes $\varepsilon$ and $\bar\varepsilon$, respectively,  are said to be \emph{Laguerre congruent} if the corresponding families of osculating circles are related by a Laguerre transformation, that is, if (up to reparameterization) $\bar\varepsilon=\lambda A\varepsilon+\vec{t}$ for some $\lambda\neq 0$, $A\in O_1(3)$, and  $\vec{t}\in \mathbf{E}_1^3$.
\end{defn}

The identification $\mathbf{E}^2\cong\{(u_1,u_2,u_3)\in \mathbf{E}_1^3|\,u_3=0\}$  induces a natural embedding of  the group $\mathbf{Iso}^+(2)=\mathbb{R}^2\rtimes SO(2)$ of all rigid motions of $\mathbf{E}^2$ in the group $\mathcal{L}$ of Laguerre transformations. The subgroup of  $\mathcal{L}$ generated by $\mathbf{Iso}^+(2)$ together with the translation group of $\mathbf{E}_1^3$ will be denoted by  $\mathcal{L}_I$.
We point out that if $L\in \mathcal{L}_I$ corresponds to a translation along the timelike axis $\vec{e}_3=(0,0,1)$, that is, $L(\vec{u})=\vec{u}+\alpha \vec{e}_3$ for some real number $\alpha$,  then
the projections of $\varepsilon$ and $\bar \varepsilon= L(\varepsilon)$ into the Euclidean plane $\mathbf{E}^2$ coincide, which implies that $\gamma$ and $\bar\gamma$  are involutes of the same curve and, consequently, they are parallel:
$\bar\gamma=\gamma+\alpha \mathbf{n}$ where $\mathbf{n}$ is the unit normal vector of $\bar \gamma$.

We also have the following.
 \begin{thm}\label{homoex1}
Two plane curves $\gamma$ and $\bar{\gamma}$ are Laguerre congruent if, and only if, the pseudo-torsions $\tau$
and $\bar{\tau}$ of $\varepsilon$ and $\bar\varepsilon$, respectively, are related by  \eqref{pseudodilation} for some $\lambda\neq 0$.
\end{thm}
\begin{proof}
Taking into account that the pseudo-torsion and the pseudo-arc parameter are invariant under Lorentz isometries and that, for dilations, \eqref{pseudodilation} holds, the assertion follows from the fact that the pseudo-torsion determines the null curve up to Lorentz isometry, as observed in Section \ref{secnul}.
\end{proof}

\subsection{The Tait theorem for osculating circles of a plane curve}

The correspondence between null curves and curves in the Euclidean plane allows one to relate an old theorem by P.G. Tait on the osculating circles of a plane curve and the following property for null curves observed by L.K. Graves.
  \begin{prop}[Graves, \cite{graves}]
   A null curve $\varepsilon$ starting at $P$ lies in the inside of the light cone $C_P$.
 \end{prop}
If  $\varepsilon=(\varepsilon_1,\varepsilon_2,\varepsilon_3)$ is a null curve with $\varepsilon(t_0)=P$, then either $\varepsilon(t)$ lies in the inside  the upper part of the light cone $C_{P}$ for all $t>t_0$ or  $\varepsilon(t)$ lies in the inside  of the lower part of the light cone $C_{P}$ for all $t>t_0$. Consequently, in both cases,  the circle associated to $P=\varepsilon(t_0)$ under the isotropic projection  does not intersect the circle associated to $\varepsilon(t)$ for all $t>0$ (see Figure \ref{tait}). This implies the following theorem.
\begin{figure}[!htb]
\centering
\includegraphics[width=7.5cm,height=4.5cm]{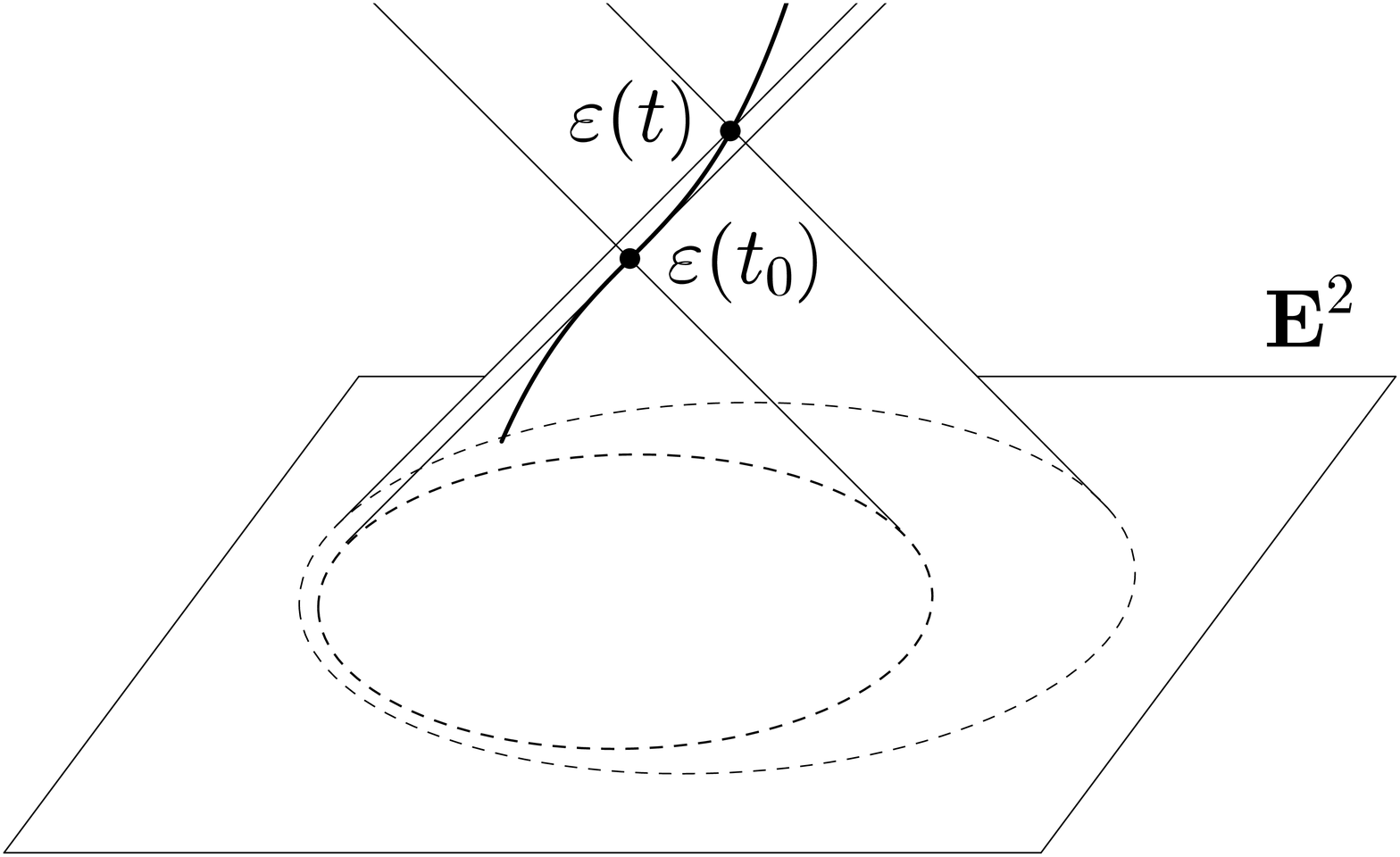}
\caption{}
\label{tait}
\end{figure}

\begin{prop}[Tait,\cite{tait}]
   The osculating circles of a curve with monotonic positive curvature
are pairwise disjoint and nested.
\end{prop}
For some variations on the P.G. Tait result, see \cite{osculating}.

\subsection{The potential function} Let $\gamma$ be a regular  curve in $\mathbf{E}^2$  with arc length parameter $t$ and (signed) radius of curvature $u$.
 Observe that the sign of $u\dot u$ is changed if the orientation of $\gamma$ is reversed.
\begin{defn}\label{defpotential}
Take an arc length parameter $t$ of $\gamma$ such that $u\dot u>0$. In \eqref{ts}, choose $$\mathrm{sign}(\dot \phi)=\mathrm{sign}(\dot{u})=\mathrm{sign}(u)$$ and let $s$ be the corresponding pseudo-arc parameter of the $L$-evolute of $\gamma$.  The \emph{potential function} of $\gamma$ is the (positive) function
  \begin{equation}\label{potential}
  f(s)=\dot\phi(s)\dot{u}(\phi(s))=\sqrt{u(\phi(s))\dot{u}(\phi(s))}.
  \end{equation}
\end{defn}

\begin{prop}
 If $f$ is the potential function of $\gamma$, the pseudo-torsion $\tau$ of $\varepsilon$ is given by
  \begin{equation}\label{pseudotorsionf}
\tau=\frac{1}{f}\ddot f-\frac{1}{2f^2}\big(\dot f^2+1 \big).
\end{equation}
\end{prop}
\begin{proof}
In view of \eqref{n1}, \eqref{n2} and \eqref{n3},  the tangent and normal vectors of $\upsilon=\varepsilon\circ \phi$ are given by
\begin{equation}\label{tn}
\mathbf{T}(s)=f(s)(\mathbf{n}(\phi(s)),1),\quad \mathbf{N}=\dot f(s) (\mathbf{n}(\phi(s)),1)-(\mathbf{t}(\phi(s)),0).
\end{equation}
From this one can check that that the binormal vector is given by
\begin{equation}\label{bb}
\mathbf{B}(s)=\Big(-\frac{\dot f(s)}{f(s)}\mathbf{t}(\phi(s))+\frac{1}{2f(s)}\{\dot f(s)^2-1\}\mathbf{n}(\phi(s)),\frac{1}{2f(s)}\{\dot f(s)^2+1\}\Big).\end{equation}
On the other hand, differentiating $\mathbf{N}$, we get
$$\dot{\mathbf{N}}(s)=\Big(\big\{\ddot{f}(s)-k(\phi(s))\dot\phi(s)\big\}\mathbf{n}(\phi(s))-\dot f(s)k(\phi(s))\dot\phi(s)\mathbf{t}(\phi(s)),\ddot f(s) \Big).$$
Observe also that
$$ k(\phi(s))\dot\phi(s)=\frac{\dot\phi(s)}{u(\phi(s))}=\frac{\dot\phi(s)\dot u(\phi(s))}{u(\phi(s))\dot u(\phi(s))}=\frac{ 1}{f(s)}.$$
Hence
$$\dot{\mathbf{N}}(s)=\Big(\big\{\ddot{f}(s)-\frac{1}{f(s)}\big\}\mathbf{n}(\phi(s))-\frac{\dot f(s)}{f(s)}\mathbf{t}(\phi(s)),\ddot f(s) \Big).$$
 Since $\tau=-\dot{\mathbf{N}}\cdot \mathbf{B}$, we deduce \eqref{pseudotorsionf}.
\end{proof}

Up to $\mathcal{L}_I$-congruence,  the  curve $\gamma$ and its $L$-evolute can be recovered from its  potential function $f$ as follows.
\begin{thm}\label{fund}
Let $f:(\alpha,\beta)\to\mathbb{R}$ be a positive and differentiable function on the open interval $(\alpha,\beta)$. Take $s_0\in (\alpha,\beta)$ and a constant $b_0$ such that $\int_{s_0}^sf(v)dv + b_0$ is nonvanishing on $(s_0,\beta)$. Set $\theta(s)=\int_{s_0}^s\frac{1}{f(v)}dv$. Then the null curve  $\varepsilon:(s_0,\beta)\to\mathbf{E}_1^3$ given by \begin{equation}\label{evolutegamma1}
      \varepsilon(s)=\int_{s_0}^s\big(\cos\theta(v),\sin\theta(v), 1 \Big)f(v)dv+(0,0,b_0)\end{equation}
has pseudo-arc parameter $s$ and $\varepsilon$ is the $L$-evolute of some regular plane curve $\gamma$ with potential function $f:(s_0,\beta)\to \mathbb{R}$. Up to a rigid motion in $\mathbf{E}^2$, this plane curve is given by
\begin{equation}\label{gamma}
      \gamma(t)=\int_{s_0}^{\phi^{-1}(t)}\big(\cos\theta(v),\sin\theta(v)  \big)\dot \phi(v)dv,\end{equation}
   where the  arc length parameter $t$ of $\gamma$ satisfies $t=\phi(s)$ for some  strictly monotone function $\phi:(s_0,\beta)\to \mathbb{R}$ with derivative
   \begin{equation}\label{parameters}
      \dot\phi(s)=\frac{\int_{s_0}^s f(v)dv+b_0}{f(s)}.
    \end{equation}

 Moreover, if $\bar\gamma$ is another regular plane curve with potential function $f:(s_0,\beta)\to \mathbb{R}$, then  $\bar\gamma$ coincides with $\gamma$ up to rigid motion, for some constant $b$. Consequently, any two plane curves with the same potential function are $\mathcal{L}_I$-congruent. Conversely,
if $\gamma$ and $\bar \gamma$ are $\mathcal{L}_I$-congruent, then they have the same potential function.
\end{thm}
\begin{proof}
Differentiating \eqref{evolutegamma1} we get
$$\dot\varepsilon (s)=f(s)(\cos\theta(s),\sin\theta(s),1).$$ From this we see that  $\dot\varepsilon\cdot \dot\varepsilon =0$, that is, $\varepsilon$ is a null curve. Differentiating again, we obtain
$$\ddot\varepsilon(s)=(-\sin\theta(s),\cos\theta(s),0)+ \dot f(s)(\cos\theta(s),\sin\theta(s),1).$$ Hence
$\ddot{\varepsilon}\cdot\ddot{\varepsilon}=1$, which means that $s$ is a pseudo-arc parameter of $\varepsilon$.

By hypothesis,  $\varepsilon_3(s)=\int_{s_0}^sf(v)dv+b_0$ is nonvanishing on $(s_0,\beta)$. Hence, by Proposition \ref{null-evolute}, $\varepsilon$ is the $L$-evolute of some plane curve. The radius of curvature $u$ of this plane curve at $s$ is precisely $\varepsilon_3(s)$. On the other hand, a simple computation shows that $t$ is an arc length parameter for the curve $\gamma$ defined by \eqref{gamma} and that the radius of curvature of $\gamma$ at $s$ is $\varepsilon_3(s)$. Hence, the fundamental theorem of plane curves assures that, up to rigid motion in $\mathbf{E}^2$, the plane curve whose $L$-evolute is  $\varepsilon$ coincides with $\gamma$.

Now, take any curve $\bar \gamma$ with potential function $f:(s_0,\beta)\to \mathbb{R}$. Let $\bar t$  and $\bar u$ be an arc length parameter and the radius of curvature, respectively, of $\bar \gamma$, so that, by definition of potential function,
\begin{equation*}\label{bar1}
f(s)=\dot{\bar{\phi}}(s)\dot{\bar {u}}(\bar\phi(s))=\sqrt{{\bar {u}}(\bar\phi(s))\dot{\bar {u}}(\bar\phi(s))},\end{equation*} with $\bar t=\bar\phi(s)$.
According to our choices in the definition of potential function, we have
\begin{equation*}\label{bar2}\frac{d\bar t}{ds}=\epsilon \sqrt{\bar u/\dot{\bar{u}}},\end{equation*} where $\epsilon:=\mathrm{sign}(\dot{\bar\phi})=\mathrm{sign}(\dot{\bar u})=\mathrm{sign}(\bar u)$.
From the first equation we see that $f={d\bar u}/{ds}$ and  multiplying both it follows that
$\frac{d\bar t}{ds}={\bar u}/{f}.$
 Hence
\begin{equation}\label{baru}
\bar{u}(\bar{\phi}(s))=\int_{s_0}^sf(v)dv+\bar{u}(\bar{\phi}(s_0)),\quad \frac{d\bar t}{ds}=\frac{\int_{s_0}^sf(v)dv+\bar{u}(\bar{\phi}(s_0))}{f(s)}.\end{equation}
Taking $b_0:={\bar {u}}(\bar\phi(s_0))$ we see from \eqref{baru} that $\gamma$ and $\bar \gamma$ have the same curvature function and the same arc length parameter $t=\bar t$, which means that $\gamma$ and $\bar{\gamma}$ are related by a rigid motion. In particular, the $L$-evolute $\bar \varepsilon$ of $\bar \gamma$ is also given by \eqref{evolutegamma1} up to  rigid motion acting on the first two coordinates. We can see this constructively as follows.

Let $\bar \varepsilon=(\bar \varepsilon_1,\bar \varepsilon_2,\bar \varepsilon_3)$ be the $L$-evolute of $\bar \gamma$. By definition of $L$-evolute, $\bar \varepsilon_3(s)=\bar{u}(\bar{\phi}(s))$. On the other hand, we know that the evolute $\gamma_{\bar\varepsilon}=(\bar\varepsilon_1,\bar\varepsilon_2)$ of $\bar\gamma$
has arc length $\bar u$ and curvature $k_\varepsilon=1/\bar{u}\dot{\bar{u}}$. Since
$$\int k_{\bar\varepsilon}d\bar u=\int \frac{1}{\bar u\dot{\bar u}}\frac{d\bar u}{ds}ds=\int\frac{1}{f}ds,$$
setting $\theta(s)=\int_{s_0}^s\frac{1}{f(v)}dv$, the curve  $\gamma_{\bar\varepsilon}$ is given, up to rigid motion, by
$$\gamma_{\bar\varepsilon}(\bar u(\bar\phi(s)))= \int_{s_0}^s\big(\cos\theta(v),\sin\theta(v)\big)f(v)dv, $$ by the fundamental theorem of plane curves. Consequently, the $L$-evolute $\bar\varepsilon$
is given, up to rigid motion acting on the first two coordinates, by
 \begin{equation*}\label{evolutegamma}
     \bar \varepsilon(s)=\int_{s_0}^s\big(\cos\theta(v),\sin\theta(v), 1 \Big)f(v)dv+(0,0,{\bar {u}}(\bar\phi(s_0))).\end{equation*}

Finally, if $\gamma$ and $\bar \gamma$ are $\mathcal{L}_I$ congruent, then their $L$-evolutes $\bar \varepsilon$ and $\varepsilon$ satisfy $\bar\varepsilon (s)=A\varepsilon(s)+\vec b$, with common pseudo-arc parameter $s$, where $A$ is a rigid motion acting on the first two coordinates and $\vec b=(0,0,b_0)\in \mathbf{E}_1^3$. Hence, the corresponding curvature radius satisfy $\bar{u}(\bar\phi(s))= {u}(\phi(s))+b$. Consequently,
$\bar f(s)=d\bar u/ds=d u/ds= f(s)$, and we are done.

%
%
\end{proof}

\begin{rem}
  These results provide a scheme to integrate equations \eqref{taudet}. Given a function $\tau(s)$, if $f(s)$ is a solution to the differential equation \eqref{pseudotorsionf}, then the null curve \eqref{evolutegamma1} has pseudo-torsion $\tau$ and pseudo-arc parameter $s$. This means that the components of the tangent vector
  \begin{equation}\label{tangent}
  \mathbf{T}=\Big(\cos\big(\int\frac{1}{f}ds\big)f,\sin\big(\int\frac{1}{f}ds\big)f,  f \Big)
  \end{equation}
  of $\varepsilon$ satisfy \eqref{taudet}. Moreover, all solutions of \eqref{pseudotorsionf} for a given  $\tau(s)$ arise in this way.

\end{rem}
\begin{rem}\label{lagcongruent}
We have  also obtained a description of all plane curves which are Laguerre congruent to a given curve $\gamma$. As a matter of fact, starting with  $\gamma$, compute its potential function and the pseudo-torsion $\tau$ of its $L$-evolute making use of \eqref{potential}  and \eqref{pseudotorsionf}; in view of Theorem \ref{homoex1},  find the general solution of the equation
$$\frac{1}{|\lambda|}\tau\big( {s}/{\sqrt{|\lambda|}}\big)=\frac{1}{f(s)}\ddot f(s)-\frac{1}{2f^2(s)}\big(\dot f^2(s)+1 \big)$$
for each $\lambda\neq 0$;
since this is a second order differential equation, the two initial conditions together with the parameter $\lambda$ determine a three-parameter family of potential functions;  for any such function $f$, formulas \eqref{gamma} and \eqref{parameters}
 define a curve in the plane which is Laguerre equivalent to $\gamma$. Conversely, any curve which is Laguerre congruent to $\gamma$ arises in this way, up to rigid motion.
\end{rem}

\begin{eg}\label{exspiral} Equation \eqref{pseudotorsionf} is equivalent to $2\tau f^2=2f\ddot f-(\dot f^2+1)$. Differentiating this, we obtain the  third order linear ordinary differential equation
\begin{equation}\label{-1}
  \dddot{f}-2\tau \dot f-\dot\tau f=0.
\end{equation}
For $\tau=-\frac{5}{2s^2}$, the general solution of \eqref{-1} is
\begin{equation}\label{potentialspiral}
  f(s)=a s+bs \sin(2\ln s)+cs\cos(2\ln s),
\end{equation}
and a straightforward computation shows that the solutions of  \eqref{pseudotorsionf}, with $\tau=-\frac{5}{2s^2}$, are precisely those functions \eqref{potentialspiral}
 satisfying $b^2+c^2-a^2=-\frac{1}{4}$. In particular, for $c=b=0$ and $a=\frac12$ we get the solution $f(s)=\frac{s}{2}$. In view of Theorem \ref{fund}, the arc length $t$ of the plane curve $\gamma$ associated to this potential function is given by $t=\frac{s^2}{4}$ and we have
\begin{equation}\label{spirallog}
\gamma(t)=\frac{t}{2}\big(\sin(\ln 4t)+\cos(\ln4t),\sin(\ln4t)-\cos(\ln4t)\big).
\end{equation}
Up to Euclidean motion, $\gamma$ is the logarithmic spiral $\theta\mapsto e^\theta(\cos\theta,\sin\theta)$ reparameterized by arc length $t$. The $L$-evolute of $\gamma$ is given by
$$
\varepsilon(t(s))=\frac{s^2}{8}\big(\sin(2\ln s)+\cos(2\ln s),\sin(2\ln s)-\cos(2\ln s),2\big).
$$
This null curve is an example of a \emph{Cartan slant helix} in $\mathbf{E}_1^3$. A Cartan slant helix in $\mathbf{E}_1^3$ is a null curve parameterized by pseudo-arc whose  normal vector makes a constant angle with a fixed direction. Accordingly to the classification established  in \cite{choikim},  Cartan slant helices are precisely those null curves whose pseudo-torsions are of the form $\pm\frac{1}{(cs+b)^2}$, where $c\neq  0$ and $b$ are constants. \end{eg}
\begin{eg}Let us consider  the \emph{Cornu's Spiral}
  $$\gamma(t)=\Big(\int_0^t\cos({v^2}/{2})dv,\int_0^t\sin({v^2}/{2})dv\Big).$$
  This is a plane curve with arc length $t$ and radius of curvature $u=1/t$. From \eqref{ts}, \eqref{potential} and \eqref{pseudotorsionf} we can see that the $L$-evolute  of $\gamma$ has pseudo-arc $s=2\sqrt{t}$, for $t>0$, the potential function is $f_\gamma(s)={8}/s^3$ and the pseudo-torsion is  $\tau=15/2s^2-s^6/128$. For this pseudo-torsion $\tau$, the general solution of \eqref{-1} is
  $$f(s)=\frac{a+b\,\sin\big(s^4/32\big)+c\,\cos\big(s^4/32\big)}{s^3};$$
  and the solutions of \eqref{pseudotorsionf} are precisely those functions $f(s)$ satisfying $a^2-b^2-c^2=64.$
\end{eg}
\begin{eg}\label{odasfiguras}The potential functions associated to the pseudo-torsion $\tau=-\frac{3}{8s^2}-\frac{1}{2s}$ are of the form
$$f(s)=a\sqrt{s}+b\sin(2\sqrt{s})\sqrt{s}+c\cos(2\sqrt{s})\sqrt{s},$$
with $a^2-b^2-c^2=1$. For $a=1$ and $b=c=0$, formula \eqref{gamma}, with $s_0=0$ and $b_0=0$, gives $\gamma(t(s))=\frac23\big(x(t(s)),y(t(s))\big)$, where
\begin{align*}
 x(t(s))&= \frac12\sqrt{s}(2s-3)\sin(2\sqrt{s})+\frac34(2s-1)\cos(2\sqrt{s})+\frac34\\
 y(t(s))&=\frac34(2s-1)\sin(2\sqrt{s})-\frac12 \sqrt{s}(2s-3)\cos(2\sqrt{s}).
\end{align*}
In the  Figure \ref{Rotulo}, the curve $\gamma$ is represented on the left, for $0<s<500$; on the right, one can see the plane curve $\bar\gamma$ associated to the potential function
 $f_{\bar\gamma}(s)=\sqrt{2s}+\sin(2\sqrt{s})\sqrt{s}$ (which corresponds to the choice $a=\sqrt2$, $b=1$ and $c=0$), obtained by numerical integration of \eqref{gamma},  with $s_0=0$ and $b_0=0$, also for $0<s<500$.
  \begin{figure}[!htb]
\centering
\includegraphics[width=4cm,height=4cm]{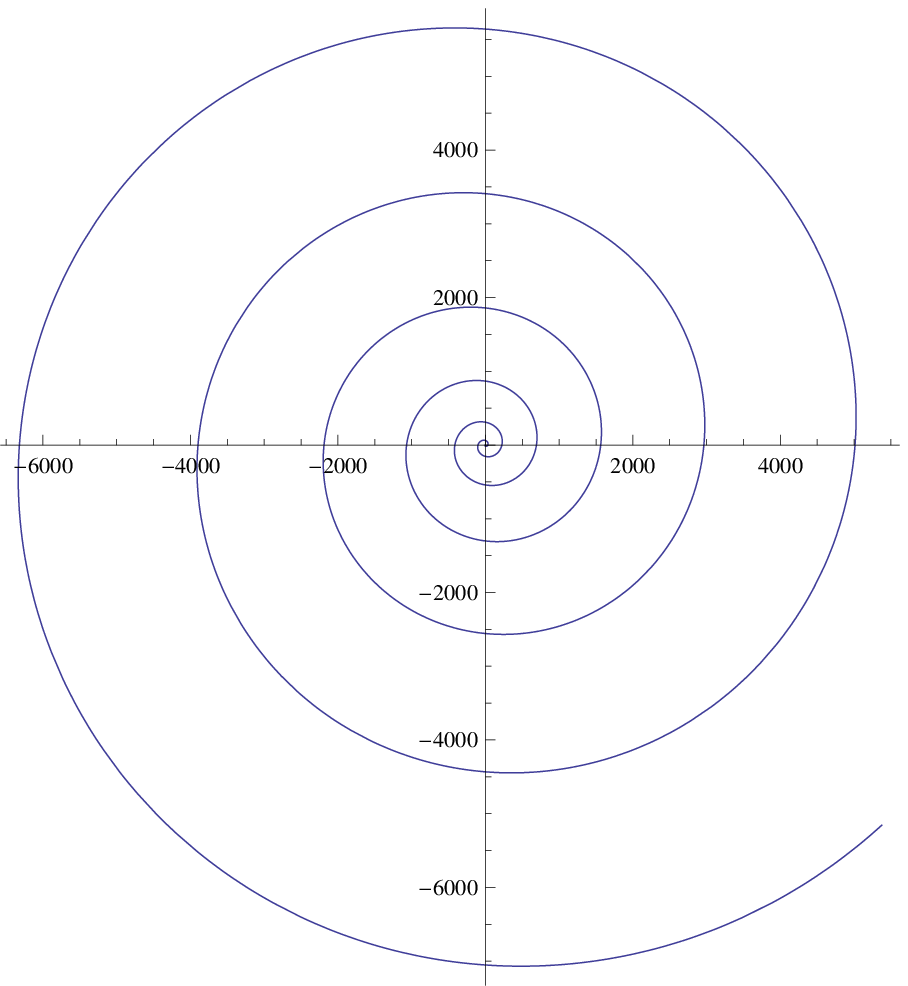}\qquad \includegraphics[width=4cm,height=4cm]{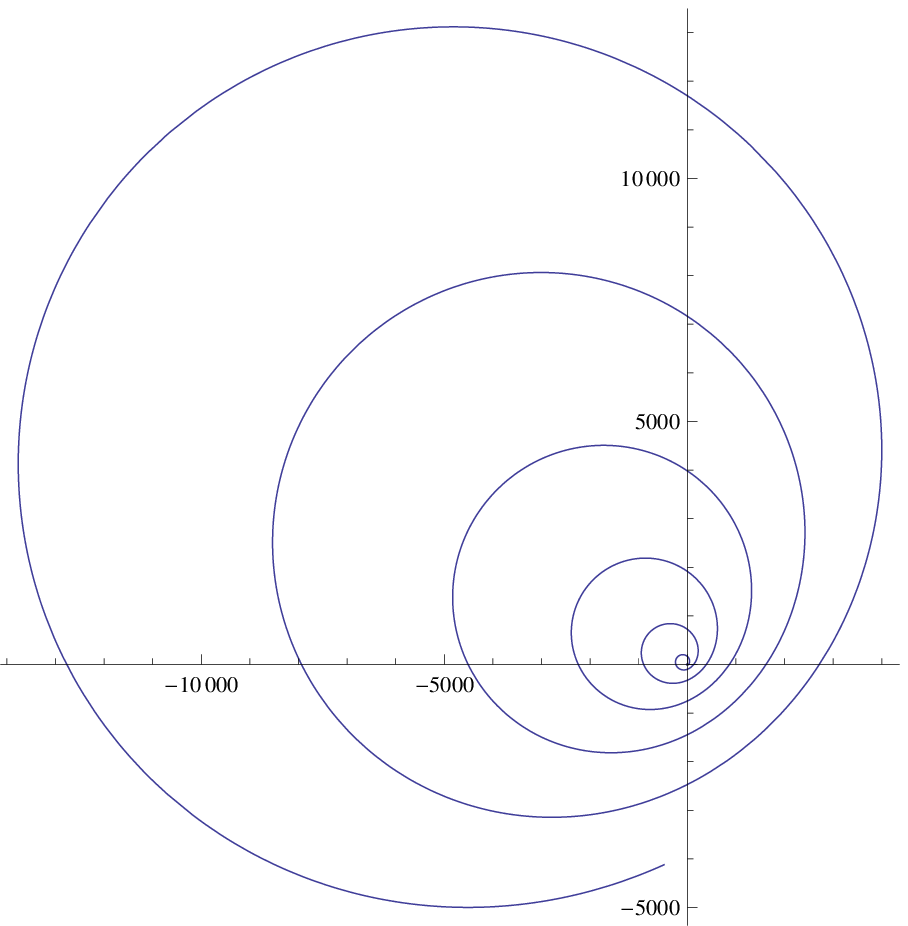}
\caption{The Laguerre congruent plane curves $\gamma$ and $\bar{\gamma}$ associated to the potential functions $f_\gamma(s)=\sqrt{s}$ and $f_{\bar\gamma}(s)=\sqrt{2s}+\sin(2\sqrt{s})\sqrt{s}$.}
\label{Rotulo}
\end{figure}
\end{eg}
\subsection{Pseudo-torsion and Schwarzian derivatives}

Very recently, Z. Olszak \cite{ols} observed that the pseudo-torsion of a null curve can be described as follows.
\begin{thm}\label{olss}\cite{ols} If $\varepsilon:I\to \mathbf{E}_1^3$ is a null curve with pseudo-arc parameter $s$, then \begin{equation}\label{ols}
\varepsilon(s)=\varepsilon(s_0)\pm\frac12\int_{s_0}^s \frac{1}{\dot g(v)}\big(2g(v),g(v)^2-1,g(v)^2+1\big)dv,
\end{equation}
with $s,s_0\in I$, for some non-zero function $g$ with nonvanishing derivative $\dot g$ on $I$. The pseudo-torsion $\tau$ of $\varepsilon$ is precisely the Schwarzian derivative of $g$:
$$\tau=S(g)=\frac{\dddot g}{\dot g}-\frac32 \Big(\frac{\ddot g}{\dot g} \Big)^2.$$
\end{thm}

Observe that \eqref{ols} can be obtained, up to Euclidean isometry in the first two coordinates, from \eqref{evolutegamma1} by using the Weierstrass substitution  \begin{equation}\label{og} g(s)=\pm \tan(\theta(s)/2)\end{equation} where $\theta(s)$ is the turning angle of the corresponding plane curve $\gamma$ at $s$.

\subsection{Potential function of the evolute}

\begin{thm}
Let $\gamma$ be a plane curve parameterized by arc length $t$ and $\gamma_\varepsilon$ be its evolute. The potential function $f$ and the pseudo-arc parameter $s$ associated to $\gamma$ are related with the potential function $f_\varepsilon$ and the pseudo-arc parameter $s_\varepsilon$ associated to the evolute $\gamma_\varepsilon$  by
\begin{equation}\label{evoluteandpotential}
 f_\varepsilon^2(s_\varepsilon(s))=2f^2(s)\Big|\frac{df}{ds}(s)\Big|
\end{equation}
and $s_\varepsilon=\beta(s)$ with $\beta(s)=\int_{s_0}^s\sqrt{2|\frac{df}{ds}(v)|}dv+s_\varepsilon(s_0).$
Consequently, the  $L$-evolute of $\gamma_\varepsilon$ is given by
 \begin{equation} \label{evoevo}
  \varepsilon_{\gamma_\varepsilon}(\beta(s))=2\int_{s_0}^sf(v)\Big|\frac{df}{ds}(v)\Big|\big(\cos\theta(v),\sin\theta(v), 1  \big)dv,\end{equation} up to congruence in $\mathcal{L}_I$, where $\theta(s)=\int_{s_0}^s\frac{1}{f(v)}dv.$
\end{thm}
\begin{proof}
Recall that the arc length parameter of $\gamma_\varepsilon$ is precisely the radius of curvature $u$ of $\gamma$. By \eqref{curvatureevolute} and taking into account the definition of potential function, the radius of curvature of $\gamma_\varepsilon$ is given by $u_\varepsilon=u\dot u=f^2$. Hence, since $f=\frac{du}{ds}$,
$$f_\varepsilon^2=\Big|u_\varepsilon\frac{d u_\varepsilon}{du}\Big|=\Big|f^2\frac{d u_\varepsilon}{ds}\frac{d s}{du}\Big|=2f^2 \Big|\frac{d f}{ds}\Big|.$$

On the other hand, taking into account \eqref{ts} and the previous observations, we also have
$$\frac{ds_\varepsilon}{ds}=\frac{ds_\varepsilon}{du}\frac{du}{ds}=\sqrt{\Big|\frac{1}{u_\varepsilon}\frac{du_\varepsilon}{du}\Big|}\,f=\frac{f_\varepsilon}{f}=\sqrt{2\Big|\frac{df}{ds}\Big|}, $$
hence  $s_\varepsilon=\beta(s)$ with $\beta(s)=\int_{s_0}^s\sqrt{2|\frac{df}{ds}(v)|}dv+s_\varepsilon(s_0).$

Formula \eqref{evoevo} follows straightforwardly by applying  \eqref{evolutegamma1} to $f_\varepsilon$ and by taking the change of parameter $s_\varepsilon=\beta(s)$.
\end{proof}

\begin{cor}The pseudo-torsion $\tau_\varepsilon$ of $\varepsilon_{\gamma_\varepsilon}$, the $L$-evolute of $\gamma_\varepsilon$, is given by
  \begin{equation}\label{tautau}\tau_\varepsilon(s_\varepsilon=\beta(s))=\frac{\tau(s)-S(\beta(s))}{2\big|\frac{df}{ds}\big|},\end{equation}
  where $S(\beta(s))$ is the Schwarzian derivative of $\beta(s).$
\end{cor}
\begin{proof}
The turning angles of $\gamma$ and its evolute $\gamma_\varepsilon$ differ by  $\pi/2$, which implies that the corresponding functions  $g$ and $g_\varepsilon$ given by \eqref{og} satisfy $$g_\varepsilon \circ \beta(s)=\pm  \tan(\theta(s)/2+\pi/4)=\pm \frac{ \tan(\theta(s)/2)+1 }{1-\tan(\theta(s)/2)}=  h\circ g (s)$$ for some fractional linear transformation $h$. In particular, $S(g)=S(g_\varepsilon \circ \beta)$. Hence, Theorem \ref{olss} together with the chain rule for the Schwarzian derivative yield
$$\tau(s)=S(g)=S(g_\varepsilon \circ \beta)=(S(g_\varepsilon)\circ\beta)\dot\beta^2+S(\beta)=(\tau_\varepsilon\circ \beta)\dot\beta^2+S(\beta).$$
Since $\dot\beta^2(s)={2|\frac{df}{ds}|}$, we are done.
\end{proof}
\begin{eg}
  If the pseudo-arc of $\gamma_\varepsilon$ coincides with that of $\gamma$, that is $s_\varepsilon(s)=s$, then we see from \eqref{evoluteandpotential} that $f_\varepsilon=f$ and $\frac{df}{ds}=\pm\frac12$.  Hence $f(s)=\pm\frac12s+c$ for some constant $c$ and \eqref{tautau} yields $\tau_\varepsilon(s)=\tau(s)$.
  For $f(s)=\frac{s}{2}$, the plane curve $\gamma$ is the logarithmic spiral \eqref{spirallog}. More generally, each potential function $f(s)=\pm\frac12s+c$
  corresponds to a plane curve whose $L$-evolute is a Cartan slant helice with pseudo-torsion $\tau=-\frac{5}{2(\pm s+2c)^2}$.
  \end{eg}
\begin{eg}
Let $\gamma$ be the plane curve associated to the potential function $f(s)=\sqrt{s}$ (see Example \ref{odasfiguras}). Then its evolute has potential function $f_\varepsilon(\beta(s))=s^{1/4}$ and $s_\varepsilon=\beta(s)$ satisfies $\dot\beta(s)=s^{-1/4}$. Integrating this we obtain $f_\varepsilon(s_\varepsilon)=\big(\frac34 s_\varepsilon \big)^{1/3}$. The pseudo-torsion of the extended evolute $\varepsilon_{\gamma_\varepsilon}$ is then given by $\tau_\varepsilon(s_\varepsilon)=-\frac{5}{32}(3s_\varepsilon/4)^{-2}-\frac12(3s_\varepsilon/4)^{-2/3}$ and we have
$$\varepsilon_{\gamma_\varepsilon}(s_\varepsilon(s))=\big(\sqrt{s}\sin(2\sqrt{s})+\frac12\cos(2\sqrt{s}),\frac12\sin(2\sqrt{s})-\sqrt{s}\cos(2\sqrt{s}),s       \big) .$$
\end{eg}

\subsection{Null helices and the corresponding potential functions}
A null curve $\varepsilon$ parameterized by the pseudo-arc parameter is called a \emph{null helix} if its  pseudo-torsion $\tau$ is constant \cite{F-G-L,ino-lee,rafael}.
Null helices admit the following classification.
\begin{prop}\cite{F-G-L}
 A null helix with pseudo-torsion $\tau$  and parameterized by the pseudo-arc parameter $s$ is congruent to one of the following:
  \begin{enumerate}
  \item if $\tau<0$, $\varepsilon_1(s)=\frac{1}{2|\tau|}\big(\cos (\sqrt{2|\tau|}\,s),\sin (\sqrt{2|\tau|}\,s), \sqrt{2|\tau|}\,s\big);$
  \item if $\tau=0$, $\varepsilon_2(s)=\big(\frac{s^3}{4}-\frac{s}{3},\frac{s^2}{2}, \frac{s^3}{4}+\frac{s}{3}\big)$;
   \item if $\tau>0$, $\varepsilon_3(s)=\frac{1}{2\tau}\big( \sqrt{2\tau}\,s,\cosh (\sqrt{2\tau}\,s),\sinh (\sqrt{2\tau}\,s)\big).$
\end{enumerate}
\end{prop}
Next we describe the corresponding potential functions.
\begin{thm}\label{t=0}
The potential functions of plane curves whose $L$-evolutes have constant pseudo-torsion $\tau$ are precisely the following:
  \begin{enumerate}
  \item if $\tau <0$, then $f_1(s)=a\cos(\sqrt{2|\tau|}s)+b\sin(\sqrt{2|\tau|}s)+c,$ with $2|\tau|(a^2+b^2)+1=2|\tau|c^2$;
 \item if $\tau=0$, then $f_2(s)=as^2+bs+c,$ with $4ac=1+b^2$;
 \item If $\tau >0$, then $f_3(s)=ae^{\sqrt{2\tau}s}+be^{-\sqrt{2\tau}s}+c,$ with $2\tau c^2+1=8\tau ab$.
\end{enumerate}
\end{thm}
\begin{proof}
  For $\tau=0$, the general solution of \eqref{-1} is $f(s)=as^2+bs+c$; and any such function is a solution of  \eqref{pseudotorsionf} with $\tau=0$ if, and only if, $4ac=1+b^2$. The remaining cases are deduced similarly.
\end{proof}
The curve $\varepsilon_1$ corresponds to the potential function $f_1$  with $a=b=0$ and $c={1/\sqrt{2|\tau|}}$; $\varepsilon_2$ corresponds to $f_2$  with $a=3/4$, $b=0$ and $c=1/3$; $\varepsilon_3$ corresponds to $f_3$  with $a=b={1}/(2\sqrt{2\tau})$ and $c=0$.
\begin{eg}\label{fcons}
  Take the potential function $f(s)=1/\sqrt{2|\tau|}$, which corresponds to the null helix $\varepsilon_1$. The corresponding plane curve $\gamma$ is the involute of a circle. Explicitly, by \eqref{parameters}, we have $\frac{dt}{ds}=s$,  where $t$ is the arc length parameter of $\gamma$, hence we have $s=\sqrt{2t}$ for $t>0$; from  \eqref{gamma}, we conclude that $\gamma(t)=\frac{1}{|2\tau|}\big(x(t),y(t)\big)$, with
  \begin{align*}
  x(t)&=2\sqrt{|\tau|t}\sin(2\sqrt{|\tau|t})+\cos(2\sqrt{|\tau|t})\\ y(t)&=\sin(2\sqrt{|\tau|t})-2\sqrt{|\tau|t}\cos(2\sqrt{|\tau|t}).
  \end{align*}
\end{eg}

\section{Associated curves}\label{assoc}

In this section we describe, in terms of their potential function, some classes of associated null curves: Bertrand pairs  \cite{BBI,ino-lee}, null curves with common binormal direction \cite{H-I}, and binormal-directional curves \cite{choikim}.

\subsection{Bertrand pairs}
Motivated by the definition of Bertrand curve in the Euclidean space, null Bertrand curves are defined as follows.
\begin{defn}
  Let $\varepsilon:I\to\mathbf{E}_1^3$ be a null curve parameterized by pseudo-arc $s$. The curve $\varepsilon$ is a \emph{null Bertrand curve} if there exists a null curve $\bar\varepsilon:\bar I\to \mathbf{E}_1^3$ and a one-to-one differentiable correspondence $\beta:I\to\bar I$ such that, for each $s\in I$, the principal normal lines of $\varepsilon$  and $\bar\varepsilon$ at $s$ and $\beta(s)$ are equal. In this case, $\bar \varepsilon$ is called a \emph{null Bertrand mate} of $\varepsilon$ and $(\varepsilon,\bar\varepsilon)$ is a \emph{null Bertrand pair}.
\end{defn}


\begin{thm}\cite{BBI,ino-lee}
  Let  $\varepsilon:I\to\mathbf{E}_1^3$ be a null curve  parameterized by pseudo-arc $s$. The null curve $\varepsilon$ is a null Bertrand curve if, and only if, it has nonzero constant pseudo-torsion $\tau$.
  In this case, if $\bar\varepsilon:\bar I\to\mathbf{E}_1^3$ is a null Bertrand mate of $\varepsilon$, with one-to-one correspondence $\beta: I\to\bar I$, then $\bar s:=\beta(s)$ is a pseudo-arc parameter of $\bar \varepsilon$, the null curve $\bar \varepsilon$ has the same pseudo-torsion $\tau$, and $\beta$ satisfies $\dot\beta(s)=\pm1$. In particular, if $\varepsilon$ has nonzero constant pseudo-torsion $\tau$, then \begin{equation}\label{bertrand}
  \bar{\varepsilon}(s)={\varepsilon}(s)-\frac{1}{\tau}\mathbf{N}(s)
  \end{equation} defined on $I$ is a null Bertrand mate of  $\varepsilon$.
  \end{thm}
\begin{cor}
  Let ($\varepsilon,\bar{\varepsilon})$ be a null Bertrand pair satisfying \eqref{bertrand}. Let $f$ and $\bar f$ be the corresponding potential functions. Then
  \begin{equation}\label{bertrandpotentials}
    \bar f(s)=f(s)-\frac{1}{\tau}\frac{d^2f}{ds^2}(s).
  \end{equation}
\end{cor}
\begin{proof}
Differentiating \eqref{bertrand}, it follows that $\bar{\mathbf{T}}(s)={\mathbf{T}}(s)-\frac{1}{\tau}\frac{d^2\mathbf{T}}{ds^2}(s)$. In view of \eqref{tangent}, and equating the third components, we
conclude \eqref{bertrandpotentials}.
\end{proof}
\begin{cor}
  Let $\gamma$ and $\bar \gamma$ be two plane curves in $\mathbf{E}^2$ an let $\varepsilon$ and  $\bar{\varepsilon}$  be the corresponding $L$-evolutes. Assume that  $(\varepsilon,\bar{\varepsilon})$ is a null Bertrand pair with constant pseudo-torsion $\tau$ satisfying \eqref{bertrand}. Then $\gamma$ and $\bar \gamma$ are congruent in $\mathcal{L}_I$ if, and only if the potential function $f$ of $\gamma$ is  $f(s)=\frac{1}{\sqrt{2|\tau|}}$ (see example  \ref{fcons}).
 \end{cor}
 \begin{proof}
   Since $\varepsilon$ and $\bar{\varepsilon}$ have the same pseudo-arc $s$, $\gamma$ and $\bar \gamma$ are congruent in $\mathcal{L}_I$ if, and only if, $\bar f= f$. If $\tau<0$, we have, from \eqref{bertrandpotentials},
   $$\bar f(s)=-a\cos(\sqrt{2|\tau|}s)-b\sin(\sqrt{2|\tau|}s)+c, $$
   with $2|\tau|(a^2+b^2)+1=2|\tau|c^2$. Hence $\bar f= f$ if, and only if, $a=b=0$ and $c=\frac{1}{\sqrt{2|\tau|}}$.

  Similarly, if $\tau>0$, it is easy to check that we can not have $\bar f= f$.
 \end{proof}
\subsection{Null curves with common binormal lines}

\begin{thm}\label{ocontra}
    Let $\varepsilon:I\to\mathbf{E}_1^3$ be a null curve parameterized by pseudo-arc parameter $s$, with pseudo-torsion  $\tau$. Then the following assertions are equivalent:
    \begin{enumerate}
      \item[a)] there exists a null curve $\bar \varepsilon:\bar I\to\mathbf{E}_1^3$ and a one-to-one correspondence $\beta:I\to\bar I$ such that, for each $s\in I$, the binormal lines of $\varepsilon$ and $\bar \varepsilon$ are equal  at $s$ and $\beta(s)$, and  $\bar s:=\beta(s)$ is a pseudo-arc parameter of $\bar \varepsilon$;
          \item[b)] $\tau$ satisfies
  \begin{equation}\label{condi}
    a_0^2v(s)^4=\pm\big(1+v(s)\dot \tau(s)\big),
  \end{equation}
 for some nonzero constant $a_0$, where $v(s)$ satisfies
   \begin{equation}\label{ss3}
\frac{1}{v(s)}=-\frac12\int\tau^2ds.
\end{equation}
    \end{enumerate}
Moreover, in this case, the pseudo-torsion of $\bar\varepsilon$ satisfies $\bar{\tau}(\bar s=\beta(s))=\pm\tau(s)$.
\end{thm}
\begin{proof}
Assume that $\bar s=\beta(s)$ is a pseudo-arc parameter of $\bar \varepsilon$ and that $\varepsilon$ and $\bar \varepsilon$ have common binormal lines at corresponding points, that is $$\bar{\mathbf{B}}(\beta(s))=a(s)\mathbf{B}(s),$$
and
$$\bar \varepsilon(\beta(s))= \varepsilon(s)+v(s)\mathbf{B}(s),$$
for some function $v(s)\neq 0$. Differentiating this with respect to $s$, we get
\begin{equation}\label{ss1}
\dot\beta \bar{\mathbf{T}}={\mathbf{T}}+\dot v \mathbf{B}+v\tau \mathbf{N}.
\end{equation}
Taking the inner product of both terms of this equation with $\bar{\mathbf{B}}$, we obtain
\begin{equation}\label{ss2}
\dot\beta=a.
\end{equation}
Since $\bar{\mathbf{T}}$ is a lightlike vector, that is $\bar{\mathbf{T}}\cdot \bar{\mathbf{T}}=0$, it also follows from \eqref{ss1} that
\begin{equation}\label{st1}-2\dot v+v^2\tau^2=0,\end{equation} which means that \eqref{ss3} holds.

Differentiating \eqref{ss1} with respect to $s$, and taking \eqref{ss2} into account, we have
\begin{equation}\label{on}
a^3\bar{\mathbf{N}}=\big(a+av\dot\tau+2a\dot v\tau-\dot a v\tau \big)\mathbf{N}+\big(a\ddot v+av\tau-\dot a\dot v\big)\mathbf{B}+\big(av\tau^2-\dot a \big)\mathbf{T}.
\end{equation}
Since $\bar{\mathbf{N}}\cdot\bar{\mathbf{B}}=0$, the component of  $\bar{\mathbf{N}}$ along $\mathbf{T}$ must vanish, that is
\begin{equation}\label{ss4}
\dot a=av\tau^2.
\end{equation}
Hence, taking \eqref{st1} and \eqref{ss4} into account, we can rewrite \eqref{on} as
$$a^3\bar{\mathbf{N}}=a(1+v\dot\tau)\mathbf{N}+av\tau(1+\dot\tau v)\mathbf{B}.$$
Now, since $\bar{\mathbf{N}}\cdot\bar{\mathbf{N}}=1$,  from this we get
\begin{equation}\label{avt}
  a^4=(1+v\dot\tau)^2.
\end{equation}
Hence
$$\bar{\mathbf{N}}=\pm (\mathbf{N}+ v\tau\mathbf{B}).$$

By differentiation of $\bar{\mathbf{B}}(\beta(s))=a(s)\mathbf{B}(s)$ with respect to $s$, it follows that
 $$a\dot{\bar{\mathbf{B}}}=\dot a \mathbf{B}+a\tau\mathbf{N}=a\tau(\mathbf{N}+v\tau \mathbf{B})=\pm a\tau \bar{\mathbf{N}}.$$ Since the pseudo-torsion $\bar\tau$ is defined by $\dot{\bar{\mathbf{B}}}=\bar\tau\mathbf{N}$, we conclude that $\bar\tau(\beta(s))=\pm \tau(s)$. Observe also that \eqref{st1} and \eqref{ss4} imply that $\dot a/a=2\dot v/v$, and consequently $a=a_0v^2$ for some constant $a_0$. Inserting this in \eqref{avt}, we obtain  \eqref{condi}.

 Conversely, given a null curve $\varepsilon$ with pseudo-arc parameter $s$ and pseudo-torsion $\tau$, take a function $v(s)$ satisfying  \eqref{condi} and $\eqref{ss3}$ for some nonconstant $a_0$.  Consequently, $v(s)$ also satisfies \eqref{st1}. For $a=a_0v^2$, it is clear that the equalities  \eqref{ss4} and \eqref{avt} hold. Define $\zeta(s)=\varepsilon(s)+v(s)\mathbf{B}(s)$. Differentiating twice with respect to $s$, we get
 \begin{align*}
   \dot\zeta(s)=\mathbf{T}+\dot v\mathbf{B}+v\tau \mathbf{N},\qquad  \ddot\zeta(s)=(1+2\dot v\tau +v\dot\tau)\mathbf{N}+(\ddot v+v\tau)\mathbf{B}+ v\tau^2\mathbf{T}.
 \end{align*}
 From this, straightforward  computations show that $\zeta$ is a null curve and that $\big|\ddot\zeta\big|^2=a^4$. This means, taking \eqref{genpseudoarc} into account, that the parameter defined by
 $\bar s=\beta(s)$, with $\beta$ satisfying $\dot\beta=a$, is a pseudo-arc parameter of $\zeta$. Consider the corresponding reparameterization $\bar \varepsilon=\zeta\circ\beta^{-1}$. We can check as above that the tangent vector $\bar{\mathbf{T}}$ and the principal normal vector $\bar{\mathbf{N}}$ of $\bar\varepsilon$ satisfy
 \begin{equation}\label{t1a}
 \bar{\mathbf{T}}=\frac{1}{a}(\mathbf{T}+\dot v\mathbf{B}+v\tau \mathbf{N}),\quad \bar{\mathbf{N}}=\pm (\mathbf{N}+ v\tau\mathbf{B}).\end{equation}
 Since, by definition, $\bar{\mathbf{B}}$ is the unique null vector such that $\bar{\mathbf{B}}\cdot \bar{\mathbf{T}}=-1$ and $\bar{\mathbf{B}}\cdot \bar{\mathbf{N}}=0$, we conclude that $\bar{\mathbf{B}}(\beta(s))=a(s)\mathbf{B}(s).$
\end{proof}

Given a curve $\gamma:I\to\mathbf{E}^3 $ in the Euclidean space $\mathbf{E}^3$, parameterized by arc length $t$, if there exists another curve $\bar \varepsilon:\bar I\to\mathbf{E}^3$ and a one-to-one correspondence $\beta:I\to\bar I$ such that, for each $t\in I$, the binormal lines of $\gamma$ and $\bar \gamma$ are equal  at $t$ and $\beta(t)$, and  $\bar t:=\beta(t)$ is an arc length parameter of $\bar \gamma$, then both curves are plane curves, that is, their torsions vanish identically (see \cite{GD}, page 161). For null helices in the Minkowski space we have a similar result.
\begin{cor}\label{corolint}
A null helix $\varepsilon$, parameterized by pseudo-arc $s$, with constant pseudo-torsion $\tau$, admits a null curve $\bar \varepsilon$, parameterized  by pseudo-arc $\bar s=\beta(s)$, with common binormal lines at corresponding points (that is, $\bar{\mathbf{B}}(\beta(s))=a(s)\mathbf{B}(s)$) if, and only if, $\tau=0$.
\end{cor}
\begin{proof}
  If $\tau$ is constant, then, any function $v(s)$ defined by \eqref{ss3} must be of the form $$v(s)=\frac{1}{v_0-\frac12\tau^2s}$$ for some integration constant $v_0$, and \eqref{condi} implies that
  $$\frac{a_0^2}{(v_0-\frac12\tau^2s)^4}=\pm1,$$ which holds if, and only if, $\tau=0$ and $a_0=\pm v_0^2$.
\end{proof}

Since the potential function $\bar f$ of $\bar\varepsilon$ is  the third component of $\bar{\mathbf{T}}$, we see from \eqref{tn}, \eqref{bb} and  \eqref{t1a} that
$$\bar{f}\big(\bar{s}=\beta(s)\big)=\frac{1}{a_0v^2(s)}\Big\{f(s)+\frac{v^2(s)\tau^2(s)}{4f(s)}(\dot{f}^2(s)+1)+{v(s)\tau(s)}\dot{f}(s)\Big\}.$$

\begin{rem}
  In \cite{H-I}, the authors investigated pairs of null curves possessing common binormal lines. We point out that  Theorem \ref{ocontra} does not contradicts the main result in \cite{H-I} since in that paper the parameters are not necessarily the pseudo-arc parameters. If one considers that the binormals coincide at corresponding points, $\bar{\mathbf{B}}(\beta(s))=a(s)\mathbf{B}(s)$, and that $s$ and $\bar s = \beta(s)$ are precisely the pseudo-arc parameters of $\varepsilon$ and  $\bar\varepsilon$, then the additional condition \eqref{condi} is necessary.

\end{rem}
\subsection{W-directional curves}

\begin{defn}
 Let $\varepsilon$ be a null curve in $\mathbf{E}_1^3$
parameterized by the pseudo-arc
parameter $s$ and $\mathbf{W}$ a null vector field along $\varepsilon$. A null curve $\bar \varepsilon$ parameterized by pseudo-arc $\bar s=\beta(s)$ is a $\mathbf{W}$-\emph{directional  curve} of $\varepsilon$ if the tangent vector $\bar{\mathbf{T}}$ coincides with
the  vector $\mathbf{W}$  at corresponding points: $\bar{\mathbf{T}}(\beta(s))=\mathbf{W}(s)$.
\end{defn}

Let us consider first the \emph{binormal-directional  curve} of $\varepsilon$, that is, take $\mathbf{W}=\mathbf{B}$, where $\mathbf{B}$ is the binormal vector field of $\varepsilon$.
 \begin{thm}\cite{choikim}
 Let $\varepsilon$ be a null curve in $\mathbf{E}_1^3$
parameterized by the pseudo-arc
parameter $s$ with  non-zero pseudo-torsion $\tau$. Let $\bar\varepsilon$ be its null binormal-directional
curve with pseudo-torsion $\bar\tau$ and pseudo-arc parameter $\bar s=\beta(s)$. Then, $\frac{d\bar{s}}{ds}=\pm \tau (s)$ and
\begin{equation*}
  \bar \tau(\beta(s))=\frac{1}{\tau(s)}.
\end{equation*}
\end{thm}

Since $\bar{\mathbf{T}}(\beta(s))=\mathbf{B}(s)$, from \eqref{bb}  we see that the potential function of $\bar \varepsilon$ is given by
$$\bar f\big(\bar s=\beta(s)\big)=\frac{1}{2f(s)}(\dot f^2(s)+1),$$
with  $\frac{d\bar{s}}{ds}=\pm \tau (s)$.
\begin{eg}
  Consider the potential function $f(s)=s/2$. As shown in Example \ref{exspiral}, this is the potential function associated to the $L$-evolute $\varepsilon$  of the logarithmic spiral \eqref{spirallog}, which has pseudo-torsion $\tau(s)=-\frac{5}{2s^2}$. The potential function associated to the binormal-directional curve $\bar\varepsilon$ of $\varepsilon$ is then given by $\bar{f}(\beta(s))=\frac{5}{4s}$. Since  $\frac{d\bar{s}}{ds}=\pm \frac{5}{2s^2}$, we can take $\bar s=\frac{5}{2s}$. Hence $\bar f(\bar s)=\bar s/2$, that is $\bar \varepsilon$ is the $L$-evolute of a logarithmic spiral congruent to  \eqref{spirallog} in $\mathcal{L}_I$.
\end{eg}

We finish this paper with the observation that, given a null curve $\varepsilon$ parameterized by pseudo-arc, there exists a null helix $\bar\varepsilon$, parameterized by pseudo-arc, with  pseudo-torsion $\bar\tau=0$  and a $1$-$1$ correspondence between points of the two curves $\varepsilon$ and $\bar\varepsilon$ such that, at corresponding points, the  tangent lines are parallel.

\begin{thm}\label{finish}
  Let $\varepsilon$ be a null curve in $\mathbf{E}_1^3$
parameterized by the pseudo-arc
parameter $s$ with  nonzero pseudo-torsion $\tau$. Let $\lambda$ be a function whose Schwarzian derivative satisfies $S(\lambda)=-\tau/2$. Define $\bar{s}=\lambda(s)$ and $\mathbf{W}(\bar s=\lambda(s)):=\dot \lambda(s)\mathbf{T}(s)$, where $\mathbf{T}$ is the tangent vector of $\varepsilon$. Then the $\mathbf{W}$-directional curve $\bar \varepsilon$ is a null helix with pseudo-parameter $\bar s$ and pseudo-torsion $\bar\tau=0$.
\end{thm}
\begin{proof}
 Let $\bar\varepsilon$ be the $\mathbf{W}$-directional curve (unique up to translation), that is $\dot{\bar\varepsilon}(\bar s)=\mathbf{W}(\bar s)$. Since
 $$\frac{d \mathbf{W}}{d\bar s}(\lambda(s))=\frac{\ddot\lambda(s)}{\dot\lambda(s)}\mathbf{T}(s)+\dot{\mathbf{T}}(s),$$ and $s$ is a pseudo-arc parameter of $\varepsilon$, meaning that $\dot {\mathbf{T}}\cdot  \dot{\mathbf{T}}=1$, we see that we also have $\dot{\mathbf{W}}\cdot\dot{\mathbf{W}}=1$. Hence $\bar s$ is a pseudo-arc parameter of $\varepsilon$.
  Since $\mathbf{T}$ is a solution of \eqref{Tc}, a straightforward computation shows that $\mathbf{W}$ satisfies $$\frac{d^3\mathbf{W}}{d\bar s^3}=0,$$ which means that $\bar\varepsilon$ is a null helix with pseudo-torsion $\bar \tau=0$. As a matter of fact, this is a particular case of the reduction procedure of a third order linear differential equation detailed in Cartan's book   \cite{Cartan} (page 48).
\end{proof}

\end{document}